\documentclass[
        11pt,
        draft=true,
        parskip=full,
]
{amsart}

\usepackage{csquotes}

\usepackage{microtype}

\usepackage[
        doi=false,
        giveninits=false,
        maxalphanames=99,
        maxnames=4,
        style=alphabetic]
{biblatex}
\bibliography{elambdainterpol}

\usepackage{IEEEtrantools}

\usepackage{mathtools}

\usepackage{amssymb, amsmath, amsthm}

\newtheoremstyle{break}% name
  {}%         Space above, empty = `usual value'
  {}%         Space below
  {\itshape}% Body font
  {}%         Indent amount (empty = no indent, \parindent = para indent)
  {\bfseries}% Thm head font
  {.}%        Punctuation after thm head
  {\newline}% Space after thm head: \newline = linebreak
  {}%         Thm head spec

\theoremstyle{break}
\newtheorem{defn}{Definition}[section]
\newtheorem{cor}[defn]{Corollary}
\newtheorem{lem}[defn]{Lemma}
\newtheorem{prop}[defn]{Proposition}
\newtheorem{rem}[defn]{Remark}

\newtheorem{thm}[defn]{Theorem}
\newtheorem{xmpl}[defn]{Example}

\newcommand{\abs}[1]{\left| #1 \right|}
\newcommand{\CN}{\mathbb{C}}
\newcommand{\codom}{\operatorname{codom}}
\newcommand{\D}{\mathrm{d}}

\newcommand{\dom}{\operatorname{dom}}

\newcommand{\FT}{\operatorname{\mathcal{F}}}
\newcommand{\one}{\operatorname{id}}

\newcommand{\iu}{\mathrm{i}}

\newcommand{\jb}[1]{\langle #1 \rangle}
\newcommand{\N}{\mathbb{N}}
\newcommand{\C}{\mathbb{C}}

\newcommand{\norm}[1]{\left\Vert #1 \right\Vert}
\newcommand{\dnorm}[1]{{\left\vert\kern-0.25ex\left\vert\kern-0.25ex\left\vert #1 
    \right\vert\kern-0.25ex\right\vert\kern-0.25ex\right\vert}}

\newcommand{\R}{\mathbb{R}}
\newcommand{\set}[1]{\left\{ #1 \right\}}
\newcommand{\setp}[2]{\left\{ #1 \, \Big| \, #2 \right\}}
\newcommand{\scp}[2]{\left(#1, #2 \right)}
\DeclareMathOperator{\supp}{supp}

\newcommand{\Z}{\mathbb{Z}}

\usepackage{enumerate}
\usepackage{bbm}

\newcommand{\ph}{\varphi}

\newcommand{\Kll}[1]{\Biggl( #1 \Biggr)}

\newcommand{\jap}[1]{\left\langle #1 \right\rangle}

\newcommand{\sk}[2]{\left\langle #1 \,, \, #2 \right\rangle}

\newcommand{\Scl}{\mathcal{S}}
\newcommand{\Fcl}{\mathcal{F}}
\newcommand{\Lcl}{\mathcal{L}}

\usepackage[draft=false]{hyperref}

\begin{document}
\title[Exponentially Weighted Modulation Spaces]{Complex Interpolation and the Monotonicity in the Spatial Integrability Parameter of Exponentially Weighted Modulation Spaces}
\author{Leonid Chaichenets and Jan Hausmann}
\begin{abstract}
We introduce the notion of common retraction and coretraction for families of Banach spaces, formulate a framework for identifying interpolation spaces, and apply it to modulation spaces with exponential weights $E^s_{p,q}$. By constructing the domain of the common coretraction, we are able to prove $E^s_{o, q} \hookrightarrow E^s_{p, q}$ for $o \leq p$, i.e. the monotonicity in the spatial integrability parameter.
\end{abstract}
\maketitle

\section{Introduction}
Interpolation theory has been established in the late fifties of the last century with major contributions by Calderon, Gagliardo, Lions and Krejn and has since become ubiquitous in the field of function spaces and partial differential equations. Of course, nowadays many textbooks such as \cite{bergh1976}, \cite{triebel1978}, and \cite{lunardi2018} are available and typically cover the family of Besov spaces. More recent spaces, such as the modulation spaces $M_{p, q}^s$ introduced by Feichtinger in \cite{Feichtinger1983}, are usually not covered in books on interpolation theory. Unfortunately, interpolation results for such spaces are often assumed to be generally known in the community and their proofs are not explicit, because they follow the Besov space case quite closely. For example, the textbook on modulation spaces \cite[Subsection 11.3]{groechenig2001} only refers to the original publication \cite{Feichtinger1983} for interpolation results; the original publication \cite[Theorem 6.1 (D)]{Feichtinger1983}does not cover the case $p = \infty$ or $q = \infty$. In the latter case different interpolation methods are common and give different results, e.g. \cite[Theorem 2.3]{wang2007}.

In \cite{wang2006} exponentially weighted modulation spaces $E^s_{p,}$ were introduced and applied to study the nonlinear Schrödinger, Ginzburg-Landau and Navier-Stokes equations. In \cite{chen2021} Vlasov-Poisson-Fokker-Planck equation was studied in similar spaces and in \cite{feichtinger2021} modulation spaces with exponentially decaying weights were applied in the context of Navier-Stokes equations. To the best of the authors’ knowledge, no interpolation result was published for these spaces.

In the paper at hand, we introduce the notion of common retraction and coretraction for families of Banach spaces, thereby obtain a framework in which interpolation spaces can be easily identified, and apply our result to the case of exponentially weighted modulation spaces. As a by-product we obtain the monotonicity of $E_{p, q}^s$ w.r.t. the spatial integrability parameter $p$.

More precisely, our main results are as follows.
\begin{thm}
\label{thm:exp_modspace_p_embedding}
Let $d \in \N$, $s \in \R$ and $p_0, p_1, q \in [1,\infty]$ with $p_0 \leq p_1$. Then
\begin{align}
\label{eqn:exp_modspace_p_embedding}
E^s_{p_0, q} \hookrightarrow E^s_{p_1,q}. 
\end{align}
\end{thm}
Remark \ref{rem:bernstein} clarifies, how this embedding is connected to the notion of common retraction and coretraction.

\begin{thm}[Complex interpolation of $E^{s}_{p,q}$]
\label{theo:complexinterpolationAspq}
Let $d \in \N$ and $\theta \in (0, 1)$. Furthermore, let $p_0, p_1, q_1 \in [1, \infty]$, $q_0 \in [1, \infty)$ and $s_0,s_1 \in \R$. We set
\begin{align}
s = (1 - \theta)s_0 + \theta s_1, && \frac{1}{p} = \frac{1 - \theta}{p_0} + \frac{\theta}{p_1}, && \frac{1}{q} = \frac{1 - \theta}{q_0} + \frac{\theta}{q_1}, \label{eq:parametercomplexinterpolation}
\end{align}
with the convention $\frac{1}{\infty} = 0$. Then
\begin{align*}
[E^{s_0}_{p_0,q_0}(\R^d), E^{s_1}_{p_1,q_1}(\R^d)]_\theta = E^{s}_{p,q}(\R^d).
\end{align*}
with the equality in the set theoretical sense and the equivalence of the norms.
\end{thm}

\begin{rem}
   Let $d \in \N$, $s \in \R$, $\theta \in (0,1)$, $1 \leq p \leq \infty$. Then
	\begin{align}
		[E^{s}_{p,\infty}(\R^d),E^{s}_{p,\infty}(\R^d)]_\theta = E^{s}_{p,\infty}(\R^d) \label{eq:interpolationcaseinfty}
	\end{align}
	with the equality in the set theoretical sense and the equivalence of the norms.
    In general
    \begin{align*}
         [E^{s_0}_{p_0,\infty}(\R^d), E^{s_1}_{p_1,\infty}(\R^d)]_\theta \neq E^{s}_{p,\infty}(\R^d).
    \end{align*}
   There is a complex interpolation method $(\cdot,\cdot)_\theta$, where for example for modulation spaces
   \begin{align*}
       (M^{s_0}_{p_0,\infty}(\R^d),M^{s_1}_{p_1,\infty}(\R^d))_\theta = M^{s}_{p,\infty}(\R^d)
   \end{align*}
   (cf. \cite[Section 2]{wang2007})
\end{rem}

The remainder of the paper is structured as follows. In Section \ref{sec:besovmodul} we recall the definitions of Besov and modulation spaces. Subsequently, in Section \ref{sec:abstrinterpol}, we recall basic notions of interpolation theory and the complex interpolation method. Section \ref{sec:idinterpol} explains the problem of identifitation of abstract interpolation spaces and introduces the aforementioned framework in Defintion \ref{defn:comretract} and Corollary \ref{cor:interpolation}. Finally, modulation spaces with exponential weights are introduced in \ref{sec:modulexpweight} and our main Theorems \ref{thm:exp_modspace_p_embedding} and \ref{theo:complexinterpolationAspq} are proven.

\subsection*{Notation}
We denote the set of natural numbers by $\N = \set{1, 2, \ldots}$ and $\N \cup \set{0}$ by $\N_0$. Of course, $\Z$, $\R$ and $\CN$ are the sets of the integers, of the reals and of the complex numbers, respectively. We set $\R^+ \coloneqq \setp{x \in \R}{x > 0}$. On $\R^d$ and $\C^d$, $d \in \N$, we often shorten the (euclidean) norm $\norm{\cdot}_2$ to $\abs{\cdot}$ and denote the scalar product by $\scp{\cdot}{\cdot}$.

Throughout the paper we consider \emph{generalized} norms, i.e. if $(X, \norm{\cdot})$ is a normed space and $X \subseteq \mathcal{V}$ we interpret $\norm{x} = \infty$ for all $x \in \mathcal{V} \setminus X$ and have $X = \setp{x \in \mathcal{V}}{\norm{x} < \infty}$. Given two (generalized) norms $\norm{\cdot}, \dnorm{\cdot}: \mathcal{V} \to [0, \infty]$ we shall call them \emph{equivalent}, if there are constants $c, C \in (0, \infty)$ such that
\begin{align*}
c \norm{x} & \leq \dnorm{x} \leq C \norm{x}
\end{align*}
for all $x \in \mathcal{V}$. Note, that in the case that $\norm{\cdot}$ is a (classical) norm on $X = \setp{x \in \mathcal{V}}{\norm{x} < \infty}$ and $\dnorm{\cdot}$ is a (classical) norm on $Y = \setp{x \in \mathcal{V}}{\dnorm{x} < \infty}$ one has $X = Y$. For $r > 0$ and $x \in X$, where $(X, \norm{\cdot})$ is a normed space, we denote by $B_r(x) = \setp{y \in X}{\norm{x - y} < r}$ the open ball of radius $r$ around $x$. Of course, $B_r \coloneqq B_r(0)$.

We denote by $C_b(X,Y)$ the space of all bounded and continuous functions $f \colon X \to Y$ and by $S(\R^d)$ the Schwartz space. For the sake of brevity we often omit the underlying Euclidean space in the notation, e.g. we write $\Scl' = \Scl'(\R^d)$ for the space of tempered distributions on $\R^d$, $d \in \N$. The same is true for the regularity parameter $s = 0$, e.g. $M_{p, q} = M_{p, q}^0(\R^d)$ in the case of modulation spaces.

We denote the Fourier transform by $\Fcl$, its inverse by $\Fcl^{-1}$, and employ the symmetric choice of constants, i.e.
\begin{align*}
\left(\Fcl^{\pm 1} f\right)(\omega) & =
\frac{1}{\left(\sqrt{2 \pi}\right)^d} \int_{\R^d} f(x) e^{\mp \iu \scp{\omega}{x}} \D{x}.
\end{align*}

Moreover, we write $X \hookrightarrow Y$ for the fact that $X \subseteq Y$ and the inclusion map $\iota: X \to Y$ is continuous.

If $X,Y$ are normed vector spaces we denote by $\Lcl(X,Y)$ the space of all linear and continuous mappings $T \colon X \to Y$. More loosely, we will write $T \in \Lcl(X, Y)$, already if $X \subseteq \mathcal{V}$, $Y \subseteq \mathcal{W}$, and $T: \mathcal{V} \to \mathcal{W}$ is linear such that $T(X) \subseteq Y$ and $T|_X \in \Lcl(X, Y)$.
Of course, $\Lcl(X) \coloneqq \Lcl(X, X)$. If $T_i: X_i \to Y_i$ for both $i \in \set{0, 1}$ and $T_0|_{X_0 \cap X_1} = T_1|_{X_0 \cap X_1}$ we shall say that $T_0$ and $T_1$ agree (on $X_0 \cap X_1$).

In many estimates throughout the paper, if $A,B > 0$, we mean by $A \lesssim B$ that there exists a constant $C > 0$ such that $A \leq CB$. 

\section{Besov and modulation spaces} \label{sec:besovmodul}
In this section, we briefly recall the definition of Besov and modulation spaces. For comprehensive references, we refer to the textbook accounts \cite{Tri10}, \cite{Tri83}, \cite{triebel1978}, \cite{sawano2018} and \cite{groechenig2001}, respectively. First, let us consider Besov spaces. Fix any $\ph_0 \in C^\infty(\R^d)$ with $\supp(\ph_0) \subseteq \overline{B}_2$ and $\ph_0(\xi) = 1$ if $\abs{\xi} \leq 1$. We set for all $j \in \N$
\begin{align*}
\ph_j(\cdot) = \ph_0(2^{-j} \cdot) - \ph_0(2^{-j + 1} \cdot).
\end{align*}
Let us remark that $\sum_{j = 0}^\infty \ph_j \equiv 1$, $\supp(\ph_j) \subseteq \overline{B}_{2^{j+1}} \setminus B_{2^{j - 1}}$ for $j \in \N$ and therefore, for every $j, k \in \N_0$, $\supp(\ph_j) \cap \supp(\ph_k) = \emptyset$, unless $\abs{j - k} > 2$. For future use (see Example \ref{xmpl:besov}) let us define
\begin{align}
\label{eq:neighbors}
\Lambda_k \coloneqq
\begin{cases}
\set{0, 1}, & \text{for $ k = 0$}, \\
\set{k - 1, k, k + 1}, & \text{for $k > 0$},
\end{cases}
\end{align}
\emph{the set of indices neighboring $k \in \N_0$}.

Now let us define the \emph{dyadic decomposition operator} $\Delta_j \coloneqq \Fcl^{-1} \ph_j \Fcl$ for any $j \in \N_0$. Then, for any $d \in \N$, any $p, q \in [1, \infty]$ and any $s \in \R$, the Besov space $B_{p, q}^s(\R^d)$ is given by
\begin{align*}
B^s_{p,q}(\R^d) & \coloneqq \setp{f \in \Scl'(\R^n)}{\norm{f}_{B^s_{p,q}} < \infty},
\end{align*}
where
\begin{align*}
    \norm{f}_{B^s_{p,q}} & \coloneqq \left( \sum_{j = 0}^\infty 2^{sjq} \norm{\Delta_j f}_{L_p}^q\right)^{\frac{1}{q}}
\end{align*}
with the usual modification for $q = \infty$. Recall, that choosing a different $\tilde{\ph}_0$ yields the same Banach space $\tilde{B}_{p, q}^s = B_{p, q}^s$ and an equivalent norm $\norm{\cdot}_{\tilde{B}_{p, q}^s} \sim \norm{\cdot}_{B_{p, q}^s}$. \\

Let us now turn to modulation spaces. For them, the dyadic decomposition operators need to be replaced by the \emph{uniform decomposition operators} $\Box_k = \Fcl^{-1} \sigma_k \Fcl$, where, for all $k \in \Z^d$, $\sigma_k \in C^\infty$, $\supp(\sigma_k) \subseteq \overline{B}_{\sqrt{d}}(k)$ and $\sum_{k \in \Z^d} \sigma_k \equiv 1$. Moreover, the \emph{weight} $j \mapsto 2^{j s}$ is replaced by the Japanese bracket $k \mapsto \jap{k}^s = \sqrt{1 + \abs{k}^2}^s$. All in all, one has
\begin{align*}
    M^s_{p,q}(\R^d) & \coloneqq \setp{f \in \Scl'}{\norm{f}_{ M^s_{p,q}} \coloneqq \left(\sum_{k \in \Z^d} \jap{k}^{sq} \norm{\Box_k f}_{L_p}^q\right)^{1/q} < \infty}
\end{align*}
with the usual modification for $q = \infty$. Let us underline that, similar to the dyadic decomposition case, $\Box_k \Box_l = 0$, if
$\abs{k - l} > 2 \sqrt{d}$. Similarly to Equation \eqref{eq:neighbors}, we define
\begin{align}
\label{eqn:zero_neighbours}
\Lambda \coloneqq \set{k \in \Z^d \mid \abs{k} \leq 2\sqrt{d}},
\end{align}
\emph{the set of indices neighboring $0 \in \Z^d$}.

\section{Abstract Interpolation Theory} \label{sec:abstrinterpol}
In this section, we recall basic notions of the interpolation theory and the complex interpolation method. We refer to \cite{triebel1978}, \cite{bergh1976}, \cite{lunardi2018} and \cite{krein2002} for the textbook accounts. As usual, we call two Banach spaces $X_0$, $X_1$ \emph{compatible}, if and only if there exists a linear Hausdorff space $\mathcal{V}$ such that
$X_0, X_1 \hookrightarrow \mathcal{V}$. In this case, we call the pair $(X_0, X_1)$ an \emph{interpolation couple}. 

A well-known example of an interpolation couple is the case $X_0 = L_p(\R^d)$, $X_1 = L_q(\R^d)$ for some $d \in \N$ and $1 \leq p, q \leq \infty$. Here, one can take $\mathcal{V} = L_0(\R^d)$, i.e. the space of (equivalence classes of) measurable functions on $\R^d$ equipped with the topology of local convergence in measure. In fact, this topology is generated by the metric
\begin{align}
    (f,g) \mapsto d(f, g) & \coloneqq \int_{\R^d} \frac{\abs{f(x) - g(x)}}{1 + \abs{f(x) - g(x)}} e^{-x^2} \D{x}.
\end{align}
Another important example concerns the vector-valued weighted sequence spaces, i.e. for parameters $1 \leq p,q \leq \infty$, $s \in \R$, an index set $I$ and a (weight) function $\omega \colon I \to \R_+$, we consider
\begin{align}
    \ell^{s,\omega}_q(I,L_p) \coloneqq \set{(f_k)_{k \in I} \in (L_p)^{I} \, \Big\vert \, \Kll{\sum_{k \in I} \omega(k)^{qs} \norm{f_k}_{L_p}^q}^{1/q} < \infty}
\end{align}
with the usual modification for $q = \infty$. Two important examples are the \emph{exponential} and the \emph{polynomial} weights, i.e. for $I \in \set{\N_0, \Z^d}$
\begin{align}
i & \stackrel{\text{e}}{\mapsto} 2^{\abs{i}}, & i & \stackrel{\text{p}}{\mapsto} \jb{i}.
\end{align}
For $I \in \set{\N_0,\Z^d}$ and $X_i = \ell_{q_i}^{s_i, \omega}(I, L_{p_i}(\R^d))$ with $i \in \set{0, 1}$ one can choose $\mathcal{V} = \bigotimes_{z \in I} L_0(\R^d)$, i.e. the space of $I$-sequences of $L_0$-functions equipped with pointwise local convergence in measure.

We call a Banach space $X$ an \emph{intermediate space} w.r.t. the interpolation couple $(X_0, X_1)$ if
\begin{align*}
X_0 \cap X_1 \hookrightarrow X \hookrightarrow X_0 + X_1.
\end{align*}
Here, $X_0 \cap X_1$ is canonically equipped with the norm $x \mapsto \max \set{\norm{x}_{X_0}, \norm{x}_{X_1}}$, whereas
$X_0 + X_1$ is canoncially equipped with the quotient space norm
\begin{align*}
x \mapsto \inf_{\substack{x = x_0 + x_1 \\ x_0 \in X_0, x_1 \in X_1}} \norm{x_0}_{X_0} + \norm{x_1}_{X_1}.
\end{align*}
Of course, $X_0$ and $X_1$ are always intermediate spaces w.r.t. their arbitrary interpolation couple $(X_0, X_1)$.

Morphisms, i.e. structure preserving mappings, between the two interpolation couples $(X_0, X_1)$ and $(Y_0, Y_1)$ are linear maps $T: X_0 + X_1 \to Y_0 + Y_1$ such that $T \in \Lcl(X_i, Y_i)$ for both $i \in \set{0, 1}$. We write $T \in \Lcl{((X_0, X_1), (Y_0, Y_1))}$ in that case. Of course, given $T_i \in \Lcl{(X_i, Y_i)}$ for both $i \in \set{0, 1}$ one can extend them to $T \in \Lcl((X_0, X_1), (Y_0, Y_1))$ if and only if $T_0$ and $T_1$ agree on $X_0 \cap X_1$. In this case,
\begin{align*}
T(x) & \coloneqq T_0(x_0) + T_1(x_1)
\end{align*}
for any decomposition $x = x_0 + x_1 \in X_0 + X_1$.

An \emph{interpolation functor} $F$ is, loosely speaking, a method to construct intermediate spaces in a structure preserving way. More precisely, if $\hat{X} \coloneqq (X_0, X_1)$ and $\hat{Y} \coloneqq (Y_0, Y_1)$ are interpolation couples and $T \in \Lcl(\hat{X}, \hat{Y})$, then $X \coloneqq F(\hat{X})$ is an intermediate space w.r.t. $\hat{X}$, $Y \coloneqq F(\hat{Y})$ is an intermediate space w.r.t. $\hat{Y}$, and $T \in \Lcl(X, Y)$. This is referred to as the \emph{interpolation property}. We shall sometimes be more verbose and write $F(T) \coloneqq T|_{F(\hat{X})} \in \Lcl(X, Y)$.

A standard example is the \emph{complex interpolation functor} $F_\theta$, where $\theta \in (0, 1)$. One writes $(X_0, X_1) \mapsto F_\theta(X_0, X_1) \eqqcolon [X_0, X_1]_\theta$. Let us denote by
\begin{align*}
S & \coloneqq \setp{z \in \C}{0 < \operatorname{Re} z < 1}
\end{align*}
and by $\overline{S}$ its closure. For an interpolation couple $\hat{X} = (X_0, X_1)$, let $\Fcl(\hat{X})$ be the space of all bounded and continuous functions $f \colon \overline{S} \to X_0 + X_1$ such that
\begin{enumerate}[(i)]
    \item $f$ is holomorphic in $S$ and
    \item $t \mapsto f(j + it) \in C_b(\R, X_j)$ for both $j \in \set{0,1}$.
\end{enumerate}
The Banach space $\Fcl(\hat{X})$ is endowed with the norm
\begin{align*}
f \mapsto \norm{f}_{\Fcl(\hat{X})} \coloneqq
\max_{j \in \set{0,1}} \sup_{t \in \R} \norm{f(j + it)}_{X_j}.
\end{align*}
For any $\theta \in (0, 1)$ the \emph{complex interpolation space}
\begin{align*}
[X_0, X_1]_\theta & \coloneqq \setp{x \in X_0 + X_1}{\exists f \in \Fcl(\hat{X}):  f(\theta) = x}
\end{align*}
is equipped with the (quotient) norm
\begin{align*}
\norm{x}_{[X_0, X_1]_\theta} & \coloneqq \inf_{\substack{f(\theta) = x \\ f \in \Fcl(X_0,X_1)}} \norm{f}_{\Fcl(X_0,X_1)}.
\end{align*}
We shall call a Banach space $X$ an \emph{interpolation space}, if there is an interpolation functor $F$ and an interpolation couple $\hat{X} = (X_0, X_1)$ such that $X =F(\hat{X})$. Here, the equality means the equality in set theoretical sense and the equivalence of the norms.

\section{Identifying Interpolation Spaces} \label{sec:idinterpol}
Often, one can identify an (abstract) interpolation space $X$ with a concrete (i.e. a known) Banach space. The classical example is the theorem of Riesz-Thorin (cf. \cite[Example 2.11]{lunardi2018}) stating that $[L_p, L_q]_\theta = L_r$ (with equality of sets and \emph{equality} of norms), where $\theta \in (0, 1)$, $p, q \in [1, \infty]$, and
\begin{align*}
    \frac{1}{r} = \frac{1 - \theta}{p} + \frac{\theta}{q}.
\end{align*}
More generally, one has the following example.

\begin{xmpl}[{cf. \cite[Theorem 1.18.1]{triebel1978}}] \label{xmpl:interpollqs}
    Let $I$ be a countable index set, $\omega \colon I \to \R_+$, $\theta \in (0, 1)$, $s_0, s_1 \in \R$ and $p_0, p_1, q_0, q_1 \in [1, \infty]$ with $\min(q_0, q_1) < \infty$. Set
    \begin{align}
        \frac{1}{p} & = \frac{1 - \theta}{p_0} + \frac{\theta}{p_1} &
        \frac{1}{q} & = \frac{1 - \theta}{q_0} + \frac{\theta}{q_1},&
        s & = (1 - \theta)s_0 + \theta s_1. \label{eq:parameters}
    \end{align}
    Then
    \begin{align}
[\ell^{s_0,\omega}_{q_0}(I,L_{p_0}), \ell^{s_1,\omega}_{q_1}(I,L_{p_1})]_{\theta} = \ell^{s, \omega}_q(I,L_p). \label{eq:interpolationseq}
\end{align}
with the equality in the set theoretical sense and the \emph{equality} of the norms.
\end{xmpl}

An important tool for identification of interpolation spaces is the \emph{retraction-coretraction-method}. At the heart of the method is the identification of an interpolation space with a certain complemented subspace of a known interpolation space. The notions of retraction and coretraction come from the category theory. There, one calls $R$ \emph{retraction} and (a corresponding) $S$ \emph{coretraction}, if
$R \circ S = \one$. In this paper, we do not require this level of generality, but only two special cases: The first one is the case of Banach spaces $X, Y$.
If there are $R \in \Lcl(X, Y)$ and $S \in \Lcl(Y,X)$ such that
\begin{align*}
R \circ S = \one_Y
\end{align*}
we call $R$ retraction and $S$ (a corresponding) coretraction and $Y$ a \emph{retract} of $X$. The second case is the case of interpolation couples $\hat{X} = (X_0, X_1)$, $\hat{Y} = (Y_0, Y_1)$. The only change is that now $R \in \Lcl(\hat{X}, \hat{Y})$, $S \in \Lcl(\hat{Y}, \hat{X})$, and $R \circ S = \one_{Y_0 + Y_1}$. One has that $\hat{Y}$ is a retract of $\hat{X}$, if and only if $Y_j$ is a retract of $X_j$ for both $j \in \set{0,1}$ and their retractions and coretractions agree on $X_0 \cap X_1$ and on $Y_0 \cap Y_1$, respectively. For reasons which are apparent in the case of, e.g., Besov spaces and for the convenience of the readers unfamiliar with the category theory, we shall call a retraction $R$ also a \emph{synthesis operator} and a coretraction $S$ also an \emph{analysis operator}.

In the case of Banach spaces $X, Y$, a retraction $R \in \Lcl(X, Y)$, and a corresponding coretraction $S \in \Lcl(Y, X)$ the operator $P \coloneqq S \circ R \in \Lcl(\hat{X})$ is a continuous projection and therefore its image $P(X)$ is closed. By the open mapping theorem, $S: Y \to P(X)$ is an isomorphism. Hence, $Y$ is isomorphic to the complemented subspace $P(X) \subseteq X$.

The aforementioned retraction-coretraction-method is summarized in the following theorem.
\begin{thm}[Retract interpolation theorem (see {\cite[Theorem 1.2.4]{triebel1978}})]
\label{theo:Interpolationtheorem}
Let $\hat{X} = (X_0, X_1)$, $\hat{Y} = (Y_0,Y_1)$ be interpolation couples, $X$ an interpolation space of $\hat{X}$, $Y$ an intermediate space w.r.t. $\hat{Y}$, and $R \in \Lcl(\hat{X}, \hat{Y})$ be the retraction with a corresponding coretraction $S \in \Lcl(\hat{Y}, \hat{X})$. Moreover, consider an interpolation functor $F$ such that $X = F(\hat{X})$, $R \in \Lcl(X, Y)$ and $S \in \Lcl(Y, X)$. Then $Y = F(\hat{Y})$, with equality of sets but only \emph{equivalence} of norms.
\end{thm}
\begin{proof}
Due to the interpolation property of $F$, we have
\begin{align*}
S & \in \Lcl(F(\hat{Y}),F(\hat{X})), & R &\in \Lcl(F(\hat{X}), F(\hat{Y})).
\end{align*}
By assumption, $Y$ is a retract of $X$ (with the retraction $R$ and a corresponding coretraction $S$). Let us denote $P \coloneqq S \circ R$. By the above, $F(S)$ is an isomorphism onto $P(F(\hat{X})) = P(X)$ and $S|_Y$ is an isomorphism onto $P(X)$, respectively. Invoking the assumption $X = F(\hat{X})$ one obtains
\begin{align}
\label{eqn:norm_equiv}
\norm{\,\cdot\,}_{F(\hat{Y})} & \sim \norm{S\,\cdot\,}_{F(\hat{X})} \sim \norm{S\,\cdot\,}_{X} \sim \norm{\,\cdot\,}_{Y}
\end{align}
finishing the proof.
\end{proof}

\begin{rem}
Let us again emphasize that one obtains only \emph{equivalence} and not \emph{equality} of norms on $Y$ and $F(\hat{Y})$ in Theorem \ref{theo:Interpolationtheorem}. The reason for that is the invocation of the open mapping theorem in Equation \eqref{eqn:norm_equiv} and we see no easy way to avoid it.
\end{rem}

Let us demonstrate how the complex interpolation spaces of Besov spaces are identifed as Besov spaces by Theorem \ref{theo:Interpolationtheorem}.
\begin{xmpl}
\label{xmpl:besov}
Fix any $d \in \N$, $s_0, s_1 \in \R$, and $p_0, p_1, q_0, q_1 \in [1, \infty]$ such that $\min \set{q_0, q_1} < \infty$. Notice that the norm on the Besov space $B_{p, q}^s$ is given by
\begin{align*}
    \norm{f}_{B^s_{p,q}} &= \norm{(\norm{\Delta_j f}_{L_p})_{j \in \N_0}}_{\ell_q^{s,e}}.
\end{align*}
Set
\begin{align*}
\hat{X} & \coloneqq
\left(l^{s_0, \text{e}}_{q_0}(\N_0, L_{p_0}(\R^d)),
l^{s_1, \text{e}}_{q_1}(\N_0, L_{p_1}(\R^d)) \right), \\ 
\hat{Y} & \coloneqq
\left(B_{p_0, q_0}^{s_0}(\R^d), B_{p_1, q_1}^{s_1}(\R^d)\right).
\end{align*}
Furthermore, let $F$ denote the complex interpolation functor and consider any $\theta \in (0,1)$. Set
\begin{align}
    \frac{1}{q} & \coloneqq \frac{1 - \theta}{q_0} + \frac{\theta}{q_1}, & \frac{1}{p} & \coloneqq \frac{1 - \theta}{p_0} + \frac{\theta}{p_1}, & s & \coloneqq (1 - \theta)s_0 + \theta s_1. \label{eq:param}
\end{align}
By Example \ref{xmpl:interpollqs} one has
\begin{align*}
X_\theta & \coloneqq l^{s, \text{e}}_q(\N_0,L_p(\R^d)) = F(\hat{X}).
\end{align*}
We set $Y_\theta \coloneqq B^s_{p,q}(\R^d)$ and define formally
\begin{align}
\label{eqn:retcor}
S f & \coloneqq (\Delta_k f)_{k \in \N}, & R(f_k) & \coloneqq \sum_{k \in \N_0} \sum_{l \in \Lambda_k} \Delta_{l} f_k.
\end{align}
Then, by \cite[Step 7 of the proof of Theorem 2.3.2]{triebel1978}, one has that $S \in \Lcl(Y_i, X_i)$, $R \in \Lcl(X_i,Y_i)$, and $R \circ S = \operatorname{id}_{Y_i}$ for all $i \in \set{0, \theta, 1}$. In particular, $S \colon Y_\theta \to S(Y_\theta)$ is an isomorphism onto a complemented subspace of $X_\theta$. Since $X_\theta$ is an interpolation space w.r.t. $\hat{X}$, $Y_\theta$ is an intermediate space w.r.t. $\hat{Y}$. Theorem \ref{theo:Interpolationtheorem} hence implies
\begin{align*}
    \left[B_{p_0, q_0}^{s_0}(\R^d), B_{p_1, q_1}^{s_1}(\R^d) \right]_\theta  & = B^s_{p,q}(\R^d).
\end{align*}
\end{xmpl}

\begin{rem}
Let us remark that for all possible parameters $p, q, s$ the definition of $R$ and $S$ from Equation \eqref{eqn:retcor} made sense on $l_{q}^s(L^p)$ and $B_{p, q}^s$ respectively and $R$ is a retraction with the corresponding coretraction $S$ there. 

More generally, the coretraction $S$ is defined on the whole of $\Scl'$, the natural domain of the retraction is
\begin{align}
\dom R = \biggl\{(f_k)_{k \in \N_0} &\in (\Scl')^{\N_0} \, \Big\vert \, \sum_{k = 0}^\infty \sum_{l \in \Lambda_k} \Delta_{k + l} f_k \\
&\text{ converges unconditionally in } \Scl'\biggr\}, \notag
\end{align}
the compositions $R \circ S = \one_{\dom(S)}$ and $S \circ R$ make sense, as
$\codom(S) = \dom(R)$, $\codom(R) = \dom(S)$.
\end{rem}

This gives rise to the following definition. 
\begin{defn}[Common retraction and coretraction] \label{defn:comretract}
Let $I$ be a non-empty index set, $\mathcal{X} \coloneqq (X_i)_{i \in I}$ and $\mathcal{Y} \coloneqq (Y_i)_{i \in I}$ be families of Banach spaces such that $(X_i, X_j)$ and $(Y_i, Y_j)$ are interpolation couples for all $i, j \in I$. Assume that there exists a set $\mathcal{V}$ and $\mathcal{W}$ such that $X_i \subseteq \mathcal{V}$, $Y_i \subseteq \mathcal{W}$ for all $i \in I$. Then functions $R: \mathcal{V} \to \mathcal{W}$ and $S: \mathcal{W} \to \mathcal{V}$ are called \emph{common retraction and corresponding coretraction} (w.r.t to the families $\mathcal{X}$ and $\mathcal{Y}$), if
$R \circ S = \one_{\mathcal{W}}$ and $R\vert_{X_i} \in \Lcl(X_i, Y_i)$, $S\vert_{Y_i} \in \Lcl(Y_i, X_i)$ for all $i \in I$.
\end{defn}

The corollary below follows from Theorem \ref{theo:Interpolationtheorem} exactly as Example \ref{xmpl:besov}.

\begin{cor}
\label{cor:interpolation}
Let $R$ and $S$ be common retraction and corresponding coretraction w.r.t to the families $\mathcal{X} \coloneqq (X_i)_{i \in I}$ and $\mathcal{Y} \coloneqq (Y_i)_{i \in I}$. Furthermore, let $F$ be an interpolation functor such that for every $i, j \in I$ the space $F(X_i,X_j) = X_k$ for some $k \in I$. Then, $F(Y_i, Y_j) = Y_k$.
\end{cor}

Corollary \ref{cor:interpolation} yields the identification of complex interpolation spaces of modulation spaces.

\begin{xmpl} \label{xmpl:modulation}
Consider $S: \mathcal{W} \to \mathcal{V}$ and $R: \mathcal{V} \to \mathcal{W}$,
where $\mathcal{W} = \Scl'$, $S f \coloneqq (\Box_k f)_{k \in \Z^d}$,
\begin{align}
\label{eqn:domain_retraction_modspaces}
\mathcal{V} = &\setp{(f_k)_{k \in \Z^d} \in (\Scl')^{\Z^d}}
{\sum_{k \in \Z^d} \sum_{l \in \Lambda} \Box_{k + l} f_k \text{ converges unconditionally in $S'$}}, \\
& R (f_k)_{k} \coloneqq \sum_{k \in \Z^d} \sum_{l \in \Lambda} \Box_{k + l} f_k.
\end{align}

Let $d \in \N$ and set $I \coloneqq [1, \infty] \times [1, \infty] \times \R$. For any $i = (p, q, s) \in I$ let $X_i \coloneqq \ell^{s,p}_{q}(\Z^d, L_{p}(\R^d)))$ and $Y_i \coloneqq M_{p, q}^{s}(\R^d)$. Notice that
\begin{align*}
\norm{f}_{Y_i} = \norm{(\norm{S f}_{L_p})_{k \in \Z^d}}_{\ell_q^{s,p}}.
\end{align*}
For $i = (p_0, q_0, s_0) \in I$, $j = (p_1, q_1, s_1) \in I$ and $\theta \in (0, 1)$, set $k = (p, q, s)$ according to Equation \eqref{eq:param}.

Recall, that a series with values in $\CN$ is unconditionally convergent if, and only if, it is absolutely convergent. Therefore the summability condition in Equation \eqref{eqn:domain_retraction_modspaces} is equivalent to the series
\begin{align*}
\sum_{k \in \Z^d} \sum_{l \in \Lambda} \Box_{k + l} f_k
\end{align*}
being unconditionally convergent in $\Scl'$. For every $f \in \Scl'$ we have
\begin{align*}
(R \circ S)(f) = \sum_{k \in \Z^d} \sum_{l \in \Lambda} \Box_{k + l} \Box_k f = \sum_{k \in \Z^d} \Box_k f = f,
\end{align*}
where we used that $\sum_{l \in \Lambda} \Box_{k + l} = 1$ on $B_{\sqrt{d}}(k)$. The operators $R$ and $S$ are even common retraction and corresponding coretraction w.r.t. the families $(X_i)_{i \in I}$ and $(Y_i)_{i \in I}$. We refer to, e.g., \cite[Lemma 2.5, Lemma 2.12, and Lemma 2.13]{chaichenets2018} for detailed calculations. Therefore, by Example \ref{xmpl:interpollqs}, we indeed have $[X_i, X_j]_\theta = X_k$. Hence, by Corollary \ref{cor:interpolation},
\begin{align*}
    [M_{p_0, q_0}^{s_0}(\R^d),M_{p_1, q_1}^{s_1}(\R^d)]_\theta = M_{p, q}^{s}(\R^d).
\end{align*}
\end{xmpl}

\section{Modulation spaces with exponential weights} \label{sec:modulexpweight}
Modulation spaces with exponentially growing weights were introduced by Wang, Zhao and Guo in \cite{wang2006}. They consist of rather well-behaved, in particular infinitely often differentiable, functions. The case of exponentially decaying weights was first considered in \cite{feichtinger2021}. The corresponding spaces are rather rough. In particular, they are no longer contained in $\Scl'$. Instead, they form are subset of ultradistributions $\Scl_1'$, namely of the dual of the Gelfand-Shilov space of Beurling type $\Scl_1$ (see {\cite[Section 3]{feichtinger2021}}). We refer to the textbook account \cite{gelfand1968} and \cite{petersson2023} for a thorough introduction. Note that our $\Scl_1$ is $\Sigma_1^1$ in the latter reference (see \cite[Theorem 3.7]{petersson2023}).

In order to define $\Scl_1'$, we consider for all $\lambda \in \R$ and all $f \in C^\infty(\R^d)$ the semi-norms
\begin{align}
p_\lambda(f) \coloneqq \sup_{x \in \R^d} e^{\lambda \abs{x}} \abs{f(x)}, & &
q_\lambda(f) \coloneqq \sup_{\xi \in \R^d}  e^{\lambda \abs{\xi}} \abs{\Fcl f(\xi)}. \label{eq:seminorms}
\end{align}
Then
\begin{align}
\Scl_1(\R^d) \coloneqq \setp{f \in C^\infty(\R^d)}{\forall \lambda \in \R: p_\lambda(f) + q_\lambda(f) < \infty} 
\end{align}
 is a locally convex topological vector space. Observe that, e.g.,
 \begin{align}
 \label{eqn:gauss}
 x \mapsto g(x) \coloneqq \frac{1}{(\sqrt{2\pi})^d} e^{-\frac{x^2}{2}} = (\Fcl g)(x) \in \Scl_1.
 \end{align}
Observe that, by \cite[Definition 2.1]{petersson2023}, $\Scl_1$ consists of precisely those $f \in C^\infty(\R^d)$ which satisfy
 \begin{align*}
    \forall h > 0 \ \exists C_h \in \R \ \forall \alpha, \beta \in \N_0^n \colon \sup_{x \in \R^n} \abs{x^\alpha D^\beta f(x)} \leq C_h h^{\abs{\alpha + \beta}} \alpha! \beta!.
 \end{align*}
 Hence, any $f \in \Scl_1(\R^d)$ is analytic (see \cite[Preface of Section 2.2, page 172]{gelfand1968}). By Liouville’s Theorem, functions of bounded support in $\Scl_1$ must be constant and from Equation \eqref{eq:seminorms} we infer that only $f = 0$ satisfies this criterion.
 
The dual space $\Scl_1'(\R^d)$ is equipped with the weak-*-topology and the operations of translation, modulation, multiplication with functions from $\Scl_1$ and the Fourier transform are defined, as usually, via duality. For example $\Fcl(f)(\ph) \coloneqq f(\Fcl \ph)$ for all $f \in \Scl_1'$ and all $\ph \in \Scl_1$.

Notice, that $\Scl_1'$ is too big to extend the coretraction $\tilde{S}: \Scl' \to \Scl'$, $f \mapsto (\Box_k f)_k$ to. More precisely, suppose $\Box_k f \in \Scl_1'$ is defined for an $f \in \Scl'_1(\R^d)$ and $k \in \Z^d$. Then, we can apply it to the Gaussian $\ph \in \Scl_1$ from \eqref{eqn:gauss} and obtain
\begin{align*}
\sk{\Box_k f}{g} & = \sk{\Fcl^{-1} \sigma_k \Fcl f}{g} = \sk{\Fcl f}{\sigma_k \Fcl^{-1} g} = \sk{\Fcl f}{\sigma_k g}.
\end{align*}
As $\sigma_k$ has compact support, the same is true for $\sigma_k g \overset{!}{\in}{\Scl_1}$ and thus, by the above, $\sigma_k g \equiv 0$, which is clearly a contradiction.

To meaningfully extend the corectraction $\tilde{S}$, we rather start with the retraction $R$. A natural domain of $(f_k) \mapsto R (f_k) = \sum_{k \in \Z^d} \sum_{l \in \Lambda} \Box_{k + l} f_k$, is
\begin{align}
\label{eqn:domain_retraction}
\mathcal{V} & \coloneqq \setp{(f_k) \in (\Scl')^{\Z^d}}{\sum_{k \in \Z^d} \sum_{l \in \Lambda} \Box_{k + l} f_k \text{ converges unconditionally in } \Scl_1'},
\end{align}
as $f_k$ shall be in the domain of the known $\Box_k: \Scl' \to \Scl'$, $k \in \Z^d$, and the series defining $R$ shall converge unconditionally in $\Scl_1'$. For the domain of the coretraction $S$ we generalize the idea from \cite[Equation (5.3)]{feichtinger2021} and use
\begin{align}
\label{eqn:domain_coretraction}
\mathcal{W} & \coloneqq
\Biggl\{f \in \Scl_1' \, \Big| \, \exists (f_l)_{l \in \Z^d} \subseteq \Scl': 
\forall l \in \Z^d: \\
\nonumber
\supp(\Fcl f_l) & \subseteq \overline{B}_{3\sqrt{d}}(l), 
f  = \sum_{l \in \Z^d} f_l \text{ unconditionally in $\Scl_1'$} \Biggr\}.
\end{align}

\begin{prop}
\label{prop:coretraction}
Let $d \in \N$. Then for all $f = \sum_{l \in \Z^d} f_l \in \mathcal{W}$ as in \eqref{eqn:domain_coretraction} the formula
\begin{align}
\label{eqn:box_on_espq}
S f & \coloneqq \Big(\sum_{l \in \Z^d} \Box_k f_l\Big)_{k \in \Z^d}
\end{align}
leads to a well-defined mapping $S: \mathcal{W} \to \mathcal{V}$ and $S$ extends $\tilde{S}$, i.e. $(S f)(k) = \Box_k f$ for any $f \in \Scl'$ and any $k \in \Z^d$.
\end{prop}
\begin{proof}
Consider any $f \in \mathcal{W}$ with $\sum_{l \in \Z^d} f_l$ converging unconditionally in $\Scl_1'$ to $f$, $f_l \in \Scl'$, $\supp(\FT f_l) \subseteq \overline{B}_{3 \sqrt{d}}(l)$ for all $l \in \Z^d$. Then the series in \eqref{eqn:box_on_espq} is, due to the support condition on $\FT f_l$, just a finite sum in $\Scl'$,
\begin{align}
\label{eqn:coretraction_extension}
(S f)(k) & = \sum_{l \in \Lambda'} \Box_k f_{k + l}
\end{align}
for all $k \in \Z^d$, where, similarly to Equation \eqref{eqn:zero_neighbours}, we set
\begin{align*}
\Lambda' \coloneqq \setp{l \in \Z^d}{\abs{l} \leq 4\sqrt{d}}.
\end{align*}

Next we show that $S f$ is well-defined. To that end, consider another representation $f = \sum_{l \in \Z^d} \tilde{f}_l \in \Scl_1'$, where $\tilde{f}_l \in \Scl'$ with $\supp(\FT \tilde{f}_l) \subseteq B_{3 \sqrt{d}}(l)$ for all $l \in \Z^d$ and define $(h_l) \coloneqq (f_l - \tilde{f}_l)$. To show is
\begin{align*}
\sum_{l \in \Lambda'} \Box_k h_{k + l} & = 0
\end{align*}
for all $k \in \Z^d$. Fix any $\varphi \in \Scl$ and any $k \in \Z^d$. Due to the bounded supports of $\Fcl h_l$, one has
\begin{align*}
\supp\Biggl(\sum_{l \in \Lambda'} \Fcl^{-1} h_{k + l}\Biggr) \subseteq \R^d \setminus B_{\sqrt{d}}(k)
\end{align*}
and therefore, indeed
\begin{align*}
\sk{\sum_{l \in \Lambda'} \Box_k h_{k + l}}{\varphi}_{\Scl' \times S} &=
\sum_{l \in \Lambda'} 
\sk{\Fcl^{-1} h_{k + l}}{\sigma_k \Fcl \varphi}_{\Scl' \times S} = 0.
\end{align*}

We prove that $S$ is an extension of $\tilde{S}$. Observe that $\Scl' \subseteq \mathcal{W}$. To that end, consider any $f \in \Scl'$ and set $(f_k)_{k \in \Z^d} \coloneqq (\Box_k f)_{k \in \Z^d}$. Clearly, $Sf = (\Box_k f)_k$ by Equation \eqref{eqn:coretraction_extension}.

In remains to show that $Sf \in \mathcal{V}$ for every $f \in \mathcal{W}$. Indeed, due to the property of bounded supports, we have
\begin{equation}
\begin{aligned}
\label{eqn:retraction_coretration_property}
\sum_{k \in \Z^d} \sum_{l \in \Lambda} \Box_{k + l} (S f)(k) & =
\sum_{k \in \Z^d} \sum_{l \in \Lambda} \Box_{k + l} \Box_k \sum_{m \in \Lambda'} f_{k + m} \\
&= \sum_{k \in \Z^d} \sum_{m \in \Lambda'} \Box_k f_{k + m} \\
&= \sum_{k' \in \Z^d} \sum_{m' \in \Lambda'} \Box_{k' + m'} f_{k'} = \sum_{k' \in \Z^d} f_{k'} = f
\end{aligned}
\end{equation}
unconditionally in $\Scl_1'$.

This concludes the proof.
\end{proof}

The respective retraction $R$ is treated in the following Proposition.
\begin{prop}
\label{prop:retraction}
Let $d \in \N$ and $\mathcal{V}$, $\mathcal{W}$ as in Equation \eqref{eqn:domain_coretraction} and \eqref{eqn:domain_retraction}, respectively, and $S$ as in Proposition \ref{prop:coretraction}. Then the series
\begin{align}
\label{eqn:retraction}
R (f_k)_{k \in \Z^d} & \coloneqq \sum_{k \in \Z^d} \sum_{l \in \Lambda} \Box_{k + l} f_k,
\end{align}
converges unconditionally in $\Scl_1'$ and defines a mapping $R: \mathcal{V} \to \mathcal{W}$. Moreover, $R \circ S = \one_\mathcal{W}$, i.e. $R$ is a retraction with corresponding coretraction $S$.
\end{prop}

\begin{proof}
The unconditional convergence of the series in \eqref{eqn:retraction} and the fact that $R: \mathcal{V} \to \mathcal{W}$ are clear from Equations \eqref{eqn:domain_coretraction} and \eqref{eqn:domain_retraction}. The property $R \circ S = \one_\mathcal{W}$ is proven in Equation
\eqref{eqn:retraction_coretration_property}.
\end{proof}
% Let  $s \in \R$ and $p, q \in [1,\infty]$, $k \in \Z^d$ and consider any $f = \sum_{l \in \Z^d} f_l \in \mathcal{V}$ as in Equation \eqref{eqn:Z}. Then the series $\sum_{l \in \Z^d} \Box_k f_l$ is a finite sum, converges therefore unconditionally in $\Scl'$, and 

% As a consequence, the multiplication of any $f \in \mathcal{Z}$ is well defined since for two approximating sequences $(f_l)_{l \in \Z^d}, (\tilde{f}_l)_{l \in \Z^d}$ of $f$ we deduce
% \begin{align*}
%      \sigma_k \sum_{\abs{l} \leq N} \Fcl f_l - \sigma_k \sum_{\abs{l} \leq N} \Fcl \tilde{f}_l = \sigma_k \Kll{\sum_{\abs{l} \leq N} \Fcl f_l - \sum_{\abs{l} \leq N} \Fcl \tilde{f}_l } \xrightarrow{N \to \infty} 0
% \end{align*}
% Due to the unconditional convergence of $\sum_{l \in \Z^d} f_l$ to $f$ in $\Scl_1'$ we obtain
% \begin{align*}
%      \Fcl^{-1} \sigma_k \Fcl f = \lim_{N \to \infty} \Fcl^{-1} \sigma_k \sum_{\abs{l} \leq N} \Fcl f_l = \sum_{l \in \Lambda_{k}} \Fcl^{-1} \sigma_k \Fcl f_{l} = \sum_{l \in \Lambda_k} \Box_k f_l \in \Scl' \subseteq \Scl_1'.
% \end{align*}
% It follows that, $\Box_k f$ is well-defined for every $f \in \mathcal{Z}$.  

We are now in the position to define modulation spaces with exponential weight.
\begin{defn}
\label{defn:modspace}
Let $d \in \N$, $s \in \R$, $p, q \in [1,\infty]$. We define the \emph{modulation space with exponential weight} by
\begin{align}
\label{eqn:exp_modspace}
E^s_{p,q}(\R^d) \coloneqq \set{f \in \mathcal{W} \mid \norm{f}_{E_{p,q}^s} < \infty},
\end{align}
where
\begin{align}
\label{eqn:exp_modspace_norm}
\norm{f}_{E_{p,q}^s} \coloneqq \norm{(2^{s\abs{k}} \norm{\Box_{k} f}_{L^p})_{k \in \Z^d}}_{\ell_q}.
\end{align}
\end{defn}

\begin{rem}
Note that our Definition \ref{defn:modspace} is equivalent to \cite[Subsection 5.1]{feichtinger2021} by their Proposition 5.4. Therefore a different choice of the partition of unity $(\sigma_k)$ will yield the same space \eqref{eqn:exp_modspace} and a norm equivalent to \eqref{eqn:exp_modspace_norm}. Moreover, $E_{p, q}^s(\R^d)$ is a Banach space.
\end{rem}

The following lemma is frequently used, when dealing with modulation type spaces.
\begin{lem}[Bernstein Multiplier Estimate, cf. {\cite[Proposition 1.9]{wang2007}}] \label{lem:bernstein}
    Let $d \in \N$ and $p_1, p_2 \in [1,\infty]$ satisfy $p_1 \leq p_2$. Then, for every $f \in \Scl'(\R^d)$
    \begin{align*}
        \norm{\Box_k f}_{{p_2}} \lesssim_d \norm{\Box_k f}_{{p_1}},
    \end{align*}
    with the implicit constant independent of $p_1$ and $p_2$.
\end{lem}

As in the case of the usual modulation spaces one has the following statements.
\begin{prop} \label{prop:embS1}
Let $d \in \N$, $s_0, s_1, \in \R$, $p, q_0, q_1 \in [1,\infty]$ such that $\Scl_1 \leq s_0, q_0 \leq q_1$. Then 
\begin{align}
\label{eq:monoton}
\Scl_1 & \hookrightarrow E^{s_0}_{p, q_0} \hookrightarrow E^{s_1}_{p,q_1} \hookrightarrow \Scl_1'. 
\end{align} 
\end{prop}

\begin{proof}
We refer to \cite[Subsection 5.1]{feichtinger2021}.
\end{proof}

We also have the previously unnoticed embedding $E_{p_1, q}^s \hookrightarrow E_{p_2, q}$ for $p_1 \leq p_2$ from Theorem \ref{thm:exp_modspace_p_embedding}.
\begin{proof}[Proof of Theorem \ref{thm:exp_modspace_p_embedding}]
The proof follows exactly as in the usual modulation space setting, i.e. by envoking Lemma \ref{lem:bernstein}.
\end{proof}

\begin{rem} \label{rem:bernstein}
The proof of Theorem \ref{thm:exp_modspace_p_embedding} makes use of the domain of the common coretration $\mathcal{W}$. In the situation of \cite[Equation (5.3)]{feichtinger2021} Bernstein’s multiplier estimate could not be applied immediately, due to the sequence $(f_k)$ being dependent on $p$.
\end{rem}

\begin{proof}[Proof of Theorem \ref{theo:complexinterpolationAspq}]
Let $R$ and $S$ be as in Proposition \ref{prop:coretraction} and \ref{prop:retraction} respectively. Set $I \coloneqq [1,\infty] \times [1,\infty] \times \R $ and let, for $i = (\tilde{p}, \tilde{q}, \tilde{s}) \in I$, $X_i \coloneqq \ell^{\tilde{s},e}_{\tilde{q}}(\Z^d,L^{\tilde{p}}(\R^d))$ and $Y_i \coloneqq E^{\tilde{s}}_{\tilde{p},\tilde{q}}$. The claim follows by Corollary \ref{cor:interpolation}, if $R$ defined above is a common retraction with corresponding coretraction $S$, w.r.t. the families $(Y_i)_{i \in I}$ and $(X_i)_{i \in I}$.

Let us fix any index $i \coloneqq (\tilde{p}, \tilde{q}, \tilde{s}) \in I$. It remains to show that $R \in \Lcl(X_i, Y_i)$, $S \in \Lcl(Y_i, X_i)$. The last claim is trivial. For the first claim, firstly observe that
$R (f_k) \in \mathcal{W}$ for $f = (f_k) \in X_i$ due to Hölder’s inequality. To see the continuity of $R$ we observe
\begin{align*}
\norm{Rf}_{E^s_{p, q}} & = \norm{(2^{s\abs{k}} \norm{\Box_{k} R f}_p)_{k \in \Z^d}}_q \\
& \leq \norm{\Kll{2^{s\abs{k}} \sum_{l \in \Lambda} \sum_{h \in \Lambda} \norm{\Box_k \Box_{k + h} f_{k + h + l}}_{L_p}}_{k \in \Z^d}}_q \\
& \lesssim \norm{\Kll{2^{s\abs{k}} \sum_{l \in \Lambda} \sum_{k \in \Lambda} \norm{f_{k + h + l}}_{L_p}}_{k \in \Z^d}}_{\ell_q} \lesssim \norm{f}_{\ell^{s}_q(L_p)},
\end{align*}
where we used Bernstein’s Multiplier Estimate from Lemma \ref{lem:bernstein} in the penultimate step. Noting that $R \circ S = \one_{Y_i}$ by Proposition \ref{prop:retraction} completes the proof.
\end{proof}

\nocite{*}
\subsection*{Acknowledgements}
We thank Professor Peer Kunstmann from Karlsruhe Institute of Technology (KIT) for a fruitful discussion and some references.

\printbibliography
\end{document}